\documentclass[11pt]{amsart}
\usepackage{amsfonts}
\usepackage{hyperref}
\usepackage[utf8]{inputenc}
\usepackage[T1]{fontenc}
\usepackage[dvips]{graphicx}
\usepackage{amssymb}
\usepackage{amsmath}
\usepackage{dsfont}
\usepackage{color}
\usepackage{latexsym}
\usepackage{bbm}
\usepackage{color}
\usepackage{amsthm}
\usepackage{multicol}
\usepackage[top   = 2.75cm,
bottom = 2.50cm,
left   = 2.50cm,
right  = 2.00cm]{geometry}

\bibstyle{plain}
\theoremstyle{plain}
\newtheorem{theorem}{Theorem}
\newtheorem{lemma}[theorem]{Lemat}

\newtheorem{corollary}[theorem]{Corollary}

\newtheorem{definition}[theorem]{Definicja}

\theoremstyle{remark}

\newtheorem{remark}[theorem]{Uwaga}

\begin{document}

\title[]{Fractal functions defined in terms of number representations in systems with a redundant alphabet
}
\author[M.\ V.\ Pratsiovytyi and 
S.\ P.\ Ratushniak]{ M.\ V.\ Pratsiovytyi, S.\ P.\ Ratushniak, Yu.\ Yu.\ Vovk,
Ya.\ V.\  Goncharenko}

\newcommand{\eacr}{\newline\indent}

\address{M.V. Pratsiovytyi\eacr
Institute of Mathematics of NASU,
 Dragomanov Ukrainian State University, Kyiv,
Ukraine\acr ORCID 0000-0001-6130-9413}
\email{prats4444@gmail.com}

\address{S.P. Ratushniak\eacr
Institute of Mathematics of NASU,
 Dragomanov Ukrainian State University, Kyiv,
Ukraine\acr ORCID 0009-0005-2849-6233}
\email{ratush404@gmail.com}

\address{Yu.Yu. Vovk\eacr
K. D. Ushynskyi Chernihiv Regional Institute of Postgraduate Pedagogical Education}
\email{freeeidea@ukr.net}

\address{Ya.V. Goncharenko\eacr
 Dragomanov Ukrainian State University, Kyiv,
Ukraine\acr ORCID 0009-0009-2541-7631}
\email{goncharenko.ya.v@gmail.com}

\keywords{Numeral system with natural base and redundant alphabet, cylinder, achievement set (set of subsums) of a convergent series, self-similar set, fractal dimension, nowhere monotonic function, functions of unbounded variation, fractal level set of function.\\This work was supported by a grant from the Simons
Foundation (SFI-PD-Ukraine-00014586, M.P., S.R.}

\date{\today}

\newcommand{\acr}{\newline\indent}

\begin{abstract}
For fixed natural numbers $r$ and $s$, where $2\leq s \leq r$, we consider a representation of numbers from the interval $[0;\frac{r}{s-1}]$ obtained by encoding numbers by means of the alphabet $A=\{0,1,...,r\}$ via the expansion $x=\sum\limits_{n=1}^{\infty}s^{-n}\alpha_n=\Delta^{r_s}_{\alpha_1\alpha_2...\alpha_n...}.$ The algorithm for expanding a number into such a series is justified in the paper. The geometry of this representation is studied, including the geometric meaning of digits, properties of cylinder sets -- particularly the specificity of their overlaps -- and metric relations, as well as the connection between the representation and partial sums of the corresponding series.

The paper also presents results on the study of a function $f$ defined by
\[f(x=\sum\limits_{n=1}^{\infty}\frac{\alpha_n}{(r+1)^n})=\Delta^{r_s}_{\alpha_1\alpha_2...\alpha_n...}, \alpha_n\in A.\]
It is proved that the function $f$ is continuous at every point that has a unique representation in the classical numeration system with base $r+1$, and discontinuous at points having two representations. The function has unbounded variation and a self-affine graph. For $r<2s-1$, the function possesses singleton, finite, countable, and continuum level sets, including fractal ones; for $r>2s-2$, every level set is a continuum, and moreover it is fractal or anomalously fractal.
\end{abstract}

\maketitle

\section{Introduction}
Numeral system, or systems for encoding real numbers, provide a form of existence and use of a number itself~\cite{monog_2022,Schweiger,TurbPrats1992}. There exist many essentially different systems for encoding numbers~\cite{Galambos1976,monog_2022,Schweiger}, each occupying its own niche of effective application, in particular for the description of locally complicated (in the topological and metric sense) objects such as sets, functions, and probability measures~\cite{monog_1998,prats_dop(disrtibution),prats_dop(convolution),PratsMak_UMJ}, dynamical systems~\cite{Kosh,Stan}, and others.

The present paper is devoted to numeral systems with a natural base $s$ and an alphabet $A_r\equiv \{0,1,...,r\}$, whose cardinality exceeds $s$~\cite{gonmyk,MykPrats_CH,PratsMak_UMJ,Pratratushmatst,PratsVynnO}.

For the analytic description and investigation of locally complicated metric objects (sets, fractals~\cite{PanasenkoCH8,PanasenkoUMJ,PratsMak_UMJ,TurbPrats1992}, functions, random variables~\cite{monog_1998}, dynamical systems~\cite{Albeverio}, etc.), various tools are used today, among them systems for encoding real numbers~\cite{pratscherchuvovk,Rat_Buk1,Rat_Buk2}. The arsenal of such tools is constantly expanding~\cite{PratsBondar}. This toolkit includes, in particular, numeral systems with a natural base and a redundant or nonstandard digit set~\cite{Kovalenko}.

By locally complicated objects we primarily mean continuous nowhere monotone and nowhere differentiable functions~\cite{Bush,pflug,PanasenkoCH,Pratscantorprojects,Pratsiovytyi:2013:FPF,Thim,Wunder}. Such functions often possess fractal level sets, self-similar graphs, and other manifestations of fractality, both metric and structural. The general theory of such functions is still poorly developed and is enriched mainly through individual theories of prominent representatives. Hundreds of works are devoted to further investigations of classical examples of such functions (the Weierstrass function~\cite{pflug}, the Takagi function~\cite{Allart,pflug,takagi}, the Sierpi\'nski-type functions~\cite{Pratsiovytyi:2013:FPF,serpinskyi}, etc.), their generalizations and analogues.

Among nowhere monotone functions there are functions of bounded and unbounded variation~\cite{Pratsiovytyi:2022:NLO}, functions that are nowhere differentiable, as well as functions that possess a derivative almost everywhere in the sense of Lebesgue measure, in particular singular functions whose derivative is equal to zero almost everywhere~\cite{Pratsiovytyi:2011:NMS:en,Pratsiovytyi:2022:NLO}.

Analytic constructions of nowhere monotone functions are restricted to a relatively small collection of techniques and methods: the method of condensation of singularities, the IFS method~\cite{chen,pp18,PanasenkoMatst}, defining a function by a series~\cite{PratsPanas_Visnyk}, by systems of functional equations~\cite{PratsKalash}, by finite-memory automata~\cite{PrBarMas_GenerTr}, by projecting various number representations, and others.

A separate, still insufficiently studied and interesting class of functions is formed by functions defined on an interval that have a countable everywhere dense set of discontinuity points while being continuous at all remaining points. The present paper is devoted to one class of such functions. It continues the investigations initiated in~\cite{Pratratushmatst,PratsVynnO,Vas_Vovk_Prat}.

\section{Representation of numbers in system with redundant alphabet}
Let $s$ and $r$ be fixed natural numbers such that $1<s \leq r$, let $A_r\equiv \{0,1,...,r\}$ be an alphabet (set of numbers), and let
$L_r\equiv A_r\times A_r\times...$ be the space (set) of sequences of elements from the alphabet, namely
 $L_r=\{(\alpha_n):~\alpha_n\in A_r\;\forall n\in N \}$. We consider expressions of the form
\begin{equation}\label{v1}
  \frac{\alpha_1}{s}+\frac{\alpha_2}{s^2}+\ldots+\frac{\alpha_n}{s^n}+\ldots=\sum\limits_{n=1}^{\infty}s^{-n}\alpha_n\equiv
  \Delta^{r_s}_{\alpha_1\alpha_2...\alpha_n...}, (\alpha_n)\in L_r.
\end{equation}
\begin{definition}
 If the number $x$ is the sum of the series~\eqref{v1}, then this series is called its \emph{$r_s$-expansion}, and symbolic notation
  $\Delta^{r_s}_{\alpha_1\alpha_2...\alpha_n...}$ is called the \emph{$r_s$-representation}.
\end{definition}
Parentheses in the $r_s$-representation of a numbers indicate a periodic. Clearly, a numbers $0$ and $\frac{r}{s-1}$ are minimum and maximum values of expression~\eqref{v1}. More over its have single $r_s$-representation:
$0=\Delta^{r_s}_{(0)}$; $\frac{r}{s-1}=\Delta^{r_s}_{(r)}=\frac{r}{s}+\frac{r}{s^2}+...+\frac{r}{s^n}+...$

Recall that the \emph{achievement set} (set of subsums) of a convergent series
\begin{equation}u_1+u_2+ ...+u_n+...= u_1+u_2+...+u_n+r_n=S_n+r_n=r_0
\end{equation} is defined as $E(u_n)=\{x: x=\sum\limits_{n=1}^{\infty}{\varepsilon_nu_n}, (\varepsilon_n)\in L_2\},$
where the sequence $(\varepsilon_n)$ of zeros and ones ranges over the set $L_2\equiv A_2\times A_2\times...$  of all sequences of elements of a two-symbol alphabet $A_2=\{0,1\}$.

It is obvious that the set $E$ of values of expressions ~\eqref{v1} coincides with the set of subsums of the series
\begin{equation}\label{num3}
\sum\limits_{n=1}^{\infty}u_n=\underbrace{\frac{1}{s}+...+\frac{1}{s}}_{r}+...
+\underbrace{\frac{1}{s^k}+...+\frac{1}{s^k}}_{r}+...,
\end{equation}  where  $u_{rn-(r-1)}=u_{rn-(r-2)}=...=u_{rn}=\frac{1}{s^n}$. Since $u_n\leq r_n\equiv u_{n+1}+u_{n+2}+...$ for any $n\in N$ it follows from the famous Kakeya theorem~\cite{kak}, which concerns the topological and metric properties of achievement set of a absolutely convergent series, that the set $E$ of subsums of the series~\eqref{num3} is a closed interval $[0,r_0]$.

Let us give an independent constructive proof of this statement, which highlights the ``geometry'' of the $r_s$-representation and, at the same time, provides an algorithm for the decomposition of number $x\in[0;\frac{r}{s-1}]$ into the series~\eqref{v1}. 
\begin{theorem}\label{th1}
  For any $x\in[0;\frac{r}{s-1}]$ there exist a sequence $(\alpha_n)\in L_r$ such that 
  \begin{equation}\label{v2}
    x=\sum\limits_{n=1}^{\infty}s^{-n}\alpha_n\equiv\Delta^{r_s}_{\alpha_1\alpha_2...\alpha_n...}.
  \end{equation}
  \end{theorem}
 \begin{proof}
  Since $x\in [0;\frac{r}{s-1}]$ and
\[[0;\frac{r}{s-1}]=[\frac{0}{s};\frac{r}{s(s-1)}]\cup[\frac{1}{s};\frac{s-1+r}{s(s-1) }]\cup[\frac{2}{s};\frac{2(s-1)+r}{s(s-1)}]\cup\ldots
\cup[\frac{r}{s};\frac{rs}{s(s-1) }],\]
there exists (in general, not uniquely) $\alpha_1\in A_r$ such that
\[x\in[\frac{\alpha_1}{s};\frac{(s-1)\alpha_1+r}{s(s-1)}]\Leftrightarrow \frac{\alpha_1}{s}\leq x\leq \frac{2\alpha_1+r}{s(s-1)}.\]
Hence
$$0\leq x-\frac{\alpha_1}{s}\equiv x_1\leq\frac{r}{s(s-1)}.$$
If $x_1=0$, then $x=\frac{\alpha_1}{s}+\frac{0}{s^2}+\ldots+\frac{0}{s^n}+\ldots=\Delta^{r_s}_{\alpha_1(0)}$.
If $x_1\neq 0$, then $x=\frac{\alpha_1}{s}+x_1$, where $x_1\in [0;\frac{r}{s(s-1)}]$ and the procedure is repeated for $x_1$.

Since $x_1\in [0;\frac{r}{s(s-1)}]$, and
$$[0;\frac{r}{s(s-1)}]=[\frac{0}{s^2};\frac{r}{s^2{(s-1)}}]\cup
[\frac{1}{s^2};\frac{s-1+r}{s^2 {(s-1)}}]\cup\ldots\cup[\frac{r}{s^2};\frac{r(s-1)+r}{s^2(s-1)}],$$ then
 it follows that such a thing exists $\alpha_2\in A_r$, that 
\[x_1\in[\frac{\alpha_2}{s^2};\frac{2\alpha_2+r}{s^2 {(s-1)}}]\Leftrightarrow
\frac{\alpha_2}{s^2}\leq x_1\leq \frac{2\alpha_2+r}{s^2{(s-1)}}.\]
Thus $0\leq x_1-\frac{\alpha_2}{s^2}\equiv x_2\leq\frac{r}{s^2{(s-1)}}$,
and
$$x=\frac{\alpha_1}{s}+x_1=\frac{\alpha_1}{s}+\frac{\alpha_2}{s^2}+x_2, \mbox{ where } x_2\in [0;\frac{r}{s^2{(s-1)}}].$$
Proceeding inductively, after $k$ steps we obtain
\[x=\frac{\alpha_1}{s}+\frac{\alpha_2}{s^2}+\ldots+\frac{\alpha_k}{s^k}+x_k, \mbox{ where } x_k\in [0;\frac{r}{s^k{(s-1)}}].\]
If $x_k=0$,  then $x=\Delta^{r_s}_{\alpha_1\alpha_2...\alpha_k(0)}$. If $x_k\neq 0$, the decomposition is continued with $x_k$. Thus, after a finite number of steps we obtain digits $\alpha_1,...,\alpha_n$ and a remainder $x_n\in[0;\frac{r}{s^n{(s-1)}}]$. If $x_n=0$, then
$x=\Delta^{r_s}_{\alpha_1\alpha_2...\alpha_n(0)}$ otherwise the decomposition process continues indefinitely. The convergence of the procedure is guaranteed by the fact that $\frac{r}{s^n{(s-1
)}}\to 0$ as $n\to \infty$. This completes the proof.
\end{proof}
\begin{remark}
  The above proof of Theorem~\ref{th1} is largely geometric. It clarifies the geometric meaning of the digits of the $r_s$-representation obtained by the described algorithm.
 \end{remark}

\section{The number of $r_s$-representations of a number and $r_s$-rational numbers}
The redundancy of the alphabet (i.e., the use of more digits than in the classical $s$-ary numeral system) leads to the non-uniqueness of $r_s$-representations of numbers. Hence, the problem of counting $r_s$-representations is natural and relevant in many respects. This topic has been addressed in~\cite{gonmyk,MykPrats_CH,PratsVynnO}. In this paper, we refine several results.

It is evident that the equalities
$\Delta^{r_s}_{\alpha_1...\alpha_k...}=\Delta^{r_s}_{c_1...c_k...}$ and
$\sum\limits_{k=1}^{\infty}s^{-k}(\alpha_k-c_k)=0$ are equivalent.

If $\frac{a}{s^k}+\frac{b}{s^{k+1}}=\frac{c}{s^k}+\frac {d}{s^{k+1}}$, which is equivalent to $as+b=cs+d$, then the pairs $(a,b)$ and $(c,d)$ of consecutive digits in the $r_s$-representation of a number are interchangeable. We symbolically denote this by ${(a,b)}\leftrightarrow {(c,d)}.$

\begin{remark}
  The pairs ${(c,d)}\leftrightarrow {(c+1,d-s)}$ and ${(c,d)}\leftrightarrow {(c-1,d+s)}$ are interchangeable provided that $\{c,d, c+1,d-s\} \subset A_r$ and $\{c,d,c-1,d+s\} \subset A_r$, respectively.
\end{remark}
\begin{lemma}
  If $s\leq r\leq2s-1$, then the number of interchangeable pairs of digits in $r_s$-representations is given by the formula
  \begin{equation}\label{eq:N}
    l=r(r-s+1).
  \end{equation}
\end{lemma}
\begin{proof}
  1. If $s=2=r$, then there are two interchangeable pairs: $02\leftrightarrow 10$, $12\leftrightarrow 20$.

2. If $s=3=r$, then there are three interchangeable pairs: $\overline{13}\leftrightarrow\overline{20}$, $\overline{23}\leftrightarrow\overline{30}$.

  3. If $s=3=r-1$, then there are 8 interchangeable pairs: $\overline{03}\leftrightarrow\overline{10}$,
  $\overline{13}\leftrightarrow\overline{20}$, $\overline{23}\leftrightarrow\overline{30}$, $\overline{33}\leftrightarrow\overline{40}$,
  $\overline{04}\leftrightarrow\overline{11}$, $\overline{14}\leftrightarrow\overline{21}$, $\overline{24}\leftrightarrow\overline{31}$,
  $\overline{34}\leftrightarrow\overline{41}$. This can be verified by direct checking.

  In the general case, the interchangeable pairs have the form
\[(j,s+i)\leftrightarrow(j+1,i), \mbox{ where } 0\leq j\leq r-1, 0\leq i\leq r-s\]
  Here $j$ and $i$ independently take $r$ and $r-s+1$ possible values, respectively. Hence, the total number of interchangeable pairs equals $l=r(r-s+1)$. The lemma is proved.
\end{proof}
\begin{corollary}
  If $s=r$, then the number of permutations equals $r$.
\end{corollary}
Note that when $r=2s$, the following chains occur: ${(j,2s)}\leftrightarrow {(j+1,s)}\leftrightarrow {(j+2,0)}$,
$0\leq j\leq 2s-2$.
\begin{definition}
  A number that has an $r_s$-representation with period $(0)$ is called \emph{$r_s$-rational}.
Clearly, every $r_s$-rational number is rational.
\end{definition}

The number $x=\Delta^{r_s}_{c_1...c_{k-1}(s-1)}$ is $r_s$-rational, since
 $x=\Delta^{r_s}_{c_1...c_{k-1}s(0)}$.

Not every rational number is $r_s$-rational. This is illustrated by the following example. The number  $x=\Delta^{r_s}_{(r-s+1,r-s+2)}=\frac{(r-s)(s+1)+s+2}{s^2-1}$ is rational, but not $r_s$-rational, since this number has a unique $r_s$-representation~\cite{PratsVynnO}.
 \begin{lemma}
If some $r_s$-representation of a number $x$ is periodic, then $x$ is rational.
\end{lemma}
 \begin{proof}
  Indeed,
  \begin{align*}
   x=\Delta^{r_s}_{\alpha_1...\alpha_m(a_1...a_p)}&=\sum\limits_{i=1}^{m}\frac{\alpha_i}{s^{i}}+
  \frac{1}{s^m}\sum\limits_{i=1}^{p}\frac{a_i}{s^i}+\frac{1}{s^{m+p}}\sum\limits_{i=1}^{p}\frac{a_i}{s^i}+...=
  \sum\limits_{j=1}^{m}\frac{\alpha_j}{s^j}+\sum\limits_{i=1}^{p}\frac{a_i}{s^i}
  \sum\limits_{i=1}^{\infty}\frac{1}{s^{m+ip}}=\\
    &=\sum\limits_{j=1}^{m}\frac{\alpha_j}{s^j}+\frac{s^{p-m}}{s^p-1}\sum\limits_{j=1}^{p}\frac{a_j}{s^{j}}.\qedhere
  \end{align*}
\end{proof}
\begin{lemma}
If a natural number $p$ satisfies $p\leq s$ and $0\leq ps-1\leq r$, then the equality
  \begin{equation}\label{par}
    \frac{s-1}{s^k}+\frac{s-1}{s^{k+1}}=\frac{s-p}{s^k}+\frac{ps-1}{s^{k+1}}
  \end{equation}
 holds and ${(s-1,s-1)}\leftrightarrow {(s-p,ps-1)},$
  in particular ${(s-1,s-1)}\leftrightarrow {(s-2,2s-1)}$.
\end{lemma}
 Indeed, for any $k \in N$, the equality~\eqref{par} holds.
It can be shown that for $r\geq2s-1$, every $r_s$-rational number admits a continuum of distinct representations. A more general statement, proved below, is also valid.
\section{The geometry of representations: cylinder sets}
\begin{definition}
An \emph{cylinder} (or \emph{$r_s$-cylinder}) of rank $m$ with a base $c_1c_2...c_m$ is the set $\Delta^{r_s}_{c_1c_2...c_m}$ of all number $x\in[0;\frac{r}{s-1}]$, that admit $r_s$-representation $\Delta^{r_s}_{\alpha_1\alpha_2...\alpha_k...}$ such that
  $\alpha_i=c_i$, $i=\overline{1,m}$.
\end{definition}
It follows immediately from the definition that
$\Delta^{r_s}_{c_1...c_m}=\Delta^{r_s}_{c_1...c_m0}\cup\Delta^{r_s}_{c_1...c_m1}\cup...\cup\Delta^{r_s}_{c_1...c_mr}.$

The cylinder $\Delta^{r_s}_{c_1...c_m}$ is called \emph{primary} (or \emph{parent}) cylinder with respect to the cylinders $\Delta^{r_s}_{c_1...c_mi}$, $i=\overline{1,r}$.

 Two cylinders $\Delta^{r_s}_{c_1...c_mi}$ and $\Delta^{r_s}_{c_1...c_m{[i+1]}}$ are called \emph{adjacent}. Clearly,
  $$\min\Delta^{r_s}_{c_1...c_{k-1}i}<\min\Delta^{r_s}_{c_1...c_{k-1}[i+1]},
\max\Delta^{r_s}_{c_1...c_{k-1}i}<\min\Delta^{r_s}_{c_1...c_{k-1}[i+1]}.$$

1. It is easy to prove that the cylinder $\Delta^{r_s}_{c_1...c_m}$ is an interval
$[a,a+d]$, where $a=\sum\limits_{i=1}^{m}\frac{c_i}{s^i}$, $d=\frac{r}{s^{m}{(s-1)}}$.

Therefore, the length of an $r_s$-cylinder is given by $|\Delta^{r_s}_{c_1...c_m}|=\frac{r}{s^m{(s-1)}}$ and and depends only on the rank $m$ and  not on the base. Consequently, cylinders of different ranks cannot coincide.

2. Two cylinders $\Delta^{r_s}_{c_1...c_m}$ and $\Delta^{r_s}_{d_1...d_m}$ of the same rank coincide if and only if their left endpoints coincide, that is, $\sum\limits_ {i=1}^{m}s^{-i} {(c_i-d_i)=0}$, and this is possible only when the block $(d_1,...,d_m)$ can be obtained from the block  $(c_1,...,c_m)$ by a chain of substitutions of a pair of consecutive digits by an alternative pair.
  
3. Since $\frac{r}{s^m{(s-1)}}\to 0$ as $m\to \infty$, it follows that for any sequence $(c_m)\in L_r$ the equality
\[\bigcap\limits_{m=1}^{\infty}\Delta^{r_s}_{c_1...c_m}=\Delta^{r_s}_{c_1c_2...c_m...}\]
holds. Therefore, each point of the interval $[0,\frac{r}{s-1}]$  can be regarded as an $r_s$-cylinder of infinite rank.

4. Adjacent cylinders overlap, and \[\Delta^{r_s}_{c_1...c_{k-1}i}\cap \Delta^{r_s}_{c_1...c_{k-1}[i+1]}=
[\Delta^{r_s}_{c_1...c_{k-1}[i+1](0)};\Delta^{r_s}_{c_1...c_{k-1}i(r)}].\]

The length of the overlap of adjacent cylinders is $\delta\equiv \frac{r-s+1}{s^k {(s-1)}}$.

5. The ratio of the lengths of the overlap and the length of the parent cylinder is $\frac{\delta}{|\Delta^{r_s}_{c_1...c_{k-1}}|}=\frac{r-s+1}{sr}$.
\begin{lemma}
  The intersection of two adjacent cylinders of rank $k$ is a cylinder of rank $k+p$,
  $$\Delta^{r_s}_{c_1...c_{k-1}i}\cap\Delta^{r_s}_{c_1...c_{k-1}[i+1]}=
  \Delta^{r_s}_{c_1...c_{k-1}i\underbrace{r...r}_p}=\Delta^{r_s}_{c_1...c_{k-1}[i+1]\underbrace{0...0}_p},$$
if and only if $r=\frac{s^p(s-1)}{s^p-1} $.
\end{lemma}
\begin{proof}
  A cylinder of rank $k+p$ has length $\frac{r}{s^{k+p}{(s-1)}}$, whereas the intersection
  of two adjacent cylinders $\Delta^{r_s}_{c_1...c_ki}\cap\Delta^{r_s}_{c_1...c_k[i+1]}$
  has length $\frac{r-s+1}{s^k{(s-1)}}$. Hence, the intersection is a cylinder of rank $k+p$
 if only if $\frac{r-s+1}{s^k{(s-1)}}=\frac{r}{s^{k+p}{(s-1)}}$, 
  which is equivalent to $r=\frac{s^p(s-1)}{s^p-1} $.
\end{proof}
\section{Main object}
Let $\Delta^{r+1}_{\alpha_1\alpha_2...\alpha_n...}$ be the $(r+1)$-ary representation of the number $x\in{[0;\frac{s}{r-1}]}$ in the classical positional numeral system with base $r+1$, that is, $$x=\Delta^{r+1}_{\alpha_1\alpha_2...\alpha_n...}=\sum\limits_{n=1}^{\infty}\frac{\alpha_n}{(r+1)^n}.$$
Numbers that admit two $(r+1)$-representations: $\Delta^{r+1}_{\alpha_1\alpha_2...\alpha_{n-1}\alpha_n(0)}=
\Delta^{r+1}_{\alpha_1\alpha_2...\alpha_{n-1}[\alpha_n-1](r)}$, are called $(r+1)$-\emph{binary}, , while the remaining numbers, which have a unique representation, are called $(r+1)$-\emph{unary}. By agreement, for each $(r+1)$-binary number we use only the representation with the period $(0)$. This convention ensures uniqueness of representation and guarantees the well-definedness of the function $f$ given by 
  \begin{equation}\label{df}
    f(x=\Delta^{r+1}_{\alpha_1\alpha_2...\alpha_n...})=\Delta^{r_s}_{\alpha_1\alpha_2...\alpha_n...},\;\;
    f(\Delta^{r+1}_{(r)})=1.
  \end{equation}
\begin{theorem}
The function $f$ is continuous at every $(r+1)$-unary point and discontinuous at every $(r+1)$-binary point. Moreover, the jump of $f$ at a $(r+1)$-binary point of rank $m$ equals
    \[\delta_m=\frac{r-s+1}{s^m(s-1)}.\]
\end{theorem}
\begin{proof}
Consider an arbitrary $(r+1)$-unary point $x_0=\Delta^{r+1}_{\alpha_1\alpha_2...\alpha_n...}$ and $f(x_0)=\Delta^{r_s}_{\alpha_1\alpha_2...\alpha_n...}$. If $x\neq x_0$, then there exists an index $k$ such that $\alpha_k(x)\neq\alpha_k(x_0)$, while $\alpha_i(x)=\alpha_i(x_0)$ for all $i<k$. The condition $k\to\infty$ is equivalent to $x\to x_0$. To establish continuity of $f$ at $x_0$, we show that
  \[\lim\limits_{x\to x_0}|f(x)-f(x_0)|=0.\]
By the definition of $f$
  \begin{align*}
    \lim\limits_{x\to x_0}|f(x)-f(x_0)|=&
    \lim\limits_{x\to x_0}|f(\Delta^{r_s}_{\alpha_1\alpha_2...\alpha_{k-1}\alpha'_k\alpha'_{k+1}...})-
    f(\Delta^{r_s}_{\alpha_1\alpha_2...\alpha_{k-1}\alpha_k\alpha_{k+1}...})|=\\
    =&\lim\limits_{k\to\infty}|\Delta^{r_s}_{\alpha_1\alpha_2...\alpha_{k-1}\alpha'_k\alpha'_{k+1}...}-
    \Delta^{r_s}_{\alpha_1\alpha_2...\alpha_{k-1}\alpha_k\alpha_{k+1}...}|\leq\\
    \leq&\lim\limits_{k\to\infty}\frac{1}{s^{k-1}}|\Delta^{r_s}_{\alpha'_k\alpha'_{k+1}...}-
    \Delta^{r_s}_{\alpha_k\alpha_{k+1}...}|=0.
  \end{align*}
  Hence, $f$ is continuous at every $(r+1)$-unary point.

  Continuity at a $(r+1)$-binary point would require equality of the values determined by \eqref{df} for both representations, that is,
  \[f(\Delta^{r+1}_{\alpha_1\alpha_2...\alpha_{k-1}\alpha_k(0)})=
  f(\Delta^{r+1}_{\alpha_1\alpha_2...\alpha_{k-1}[\alpha_k-1](r)}) \mbox{ for any } k\in N.\]
  This is equivalent to
\[\sum\limits_{i=1}^{k-1}\frac{\alpha_i}{s^i}+\frac{\alpha_k}{s^k}=
  \sum\limits_{i=1}^{k-1}\frac{\alpha_i}{s^i}+\frac{\alpha_k-1}{s^k}+\frac{r}{s^k(s-1)},\]
  which yields
    \[\frac{\alpha_k}{s^k}=
  \frac{\alpha_k-1}{s^k}+\frac{r}{s^k(s-1)}.\]
  This equality is possible only if $s=r+1$. Therefore, $f$ is discontinuous at every $(r+1)$-binary point unless $s=r+1$.

   The jump of $f$ at the $(r+1)$-binary point $\Delta^{r+1}_{c_1...c_{m-1}c_m(0)}=\Delta^{r+1}_{c_1...c_{m-1}[c_m-1](r)}$ of rank $m$ equals
     \begin{align*}\delta_m&=\lim\limits_{k\to \infty}f(\Delta^{r+1}_{c_1...c_{m-1}[c_m-1]\underbrace{r...r}\limits_{k}(0)})-
     f(\Delta^{r+1}_{c_1...c_{m-1}c_m(0)})=\\
     &=\lim\limits_{k\to\infty}
     (\frac{r}{s^{m+1}}+...+\frac{r}{s^{m+k}})-\frac{1}{s^m}
     =\frac{r}{s^{m}(s-1)}-\frac{1}{s^m}=\frac{r-s+1}{s^m(s-1)}.\qedhere
     \end{align*}
\end{proof}
\begin{theorem}
  The function $f$ is nowhere monotone and has unbounded variation.
\end{theorem}
\begin{proof} 
To prove non-monotonicity it suffices to show it on an arbitrary cylinder of rank $m$. Consider
  $\Delta^{r+1}_{c_1...c_m}=[\Delta^{r+1}_{c_1...c_m(0)};\Delta^{r+1}_{c_1...c_m (r)}]$
 and the points
  $x_1=\Delta^{r+1}_{c_1...c_m(0)}$,
  $x_2=\Delta^{r+1}_{c_1...c_m0r(r-1)}$,
  $x_3=\Delta^{r+1}_{c_1...c_m1(0)}$, belonging to it.
  Clearly, $x_1<x_2<x_3$. Then
  \begin{align*}
  f(x_2)-f(x_1)=&
  \frac{r}{s^{m+2}}+\frac{r-1}{s^{m+2}(s-1)}>0,\\
  f(x_3)-f(x_2)=&
  \frac{1}{s^{m+1}}-\frac{r}{s^{m+2}}-\frac{r-1}{s^{m+2}(s-1)}=\\
  =&\frac{s(s-1)-r(s-1)-r+1}{s^{m+2}(s-1)}=\\
  =&\frac{s(s-r-1)+1}{s^{m+2}(s-1)}<0.
  \end{align*}
  Hence $(f(x_2)-f(x_1))(f(x_3)-f(x_2))<0$, 
  and $f$ is non-monotone on any cylinder, and therefore nowhere monotone on its entire domain.
  
  The total variation of $f$ on $[0,1]$ is not less than the sum of magnitudes of all jump discontinuities at $(r+1)$-binary points. For rank $1$ there are $r$ such $(r+1)$-binary points, giving total jump
$\frac{r(r-s+1)}{s(s-1)}$. For rank $2$ there are $r^2$ points with total jump 
$\frac{r^2(r-s+1)}{s^2(s-1)}$, and so on. Thus, the total jump is the sum of an increasing geometric series with first term 
$\frac{r(r-s+1)}{s(s-1)}$ and ratio $q=\frac{r}{s}$, which diverges. Hence, $f$ has unbounded variation.
\end{proof}
\section{Functional relations and self-affinity of the function graph}
\begin{lemma}
  The function $f$ defined by~\eqref{df} is a solution of the system of functional equations
\[f(\frac{i+x}{r+1})=\frac{i}{s}+\frac{1}{s}f(x),\;\; i=\overline{0,r}.\]
\end{lemma}
\begin{proof}
 Let $x=\Delta^{r+1}_{\alpha_1\alpha_2...\alpha_n...}$. Then $f(x)=\Delta^{r_s}_{\alpha_1\alpha_2...\alpha_n...}.$
 We have \[\frac{i+x}{r+1}=\frac{i}{r+1}+\frac{1}{r+1}x=
  \Delta^{r+1}_{i\alpha_1\alpha_2...\alpha_n...}.\]
 and therefore
  \[f(\frac{i+x}{r+1})=
  \Delta^{r_s}_{i\alpha_1\alpha_2...\alpha_n...}.\]
On the other hand,
  \[\frac{i}{s}+\frac{i}{s}f(x)=\frac{i}{s}+\frac{1}{s}\Delta^{r_s}_{\alpha_1\alpha_2...} =\Delta^{r_s}_{i\alpha_1\alpha_2...\alpha_n...}=
  f(\frac{i+x}{r+1}),\]
  which completes the proof.
  \end{proof}

The operator of the right shift of digits in the $g$-representation of numbers is defined by the mapping
$\delta_i$
\begin{equation}\label{num3}
  \delta_i(x=\Delta^g_{\alpha_1\alpha_2...\alpha_n...})=
  \Delta^g_{i\alpha_1\alpha_2...\alpha_n...}.
\end{equation}
This definition applies to both $(r+1)$-representations and to $r_s$-representations; in the first case
the operator will be denoted by $\rho_i$.

The mapping $\delta_i$ is well defined for $r_s$-representations. Indeed, for two different $r_s$-representations $\Delta^{r_s}_{\alpha_1\alpha_2\ldots}$ and $\Delta^{r_s}_{\beta_1\beta_2\ldots}$ of the same number, we obtain
\begin{align*}
  \delta_i(\Delta^{r_s}_{\alpha_1\alpha_2...\alpha_n...})=
  &\Delta^{r_s}_{i\alpha_1\alpha_2...\alpha_n...}=
  \frac{i}{s}+\frac{1}{s}x=\frac{i}{s}+
  \frac{1}{s}\Delta^{r_s}_{\alpha_1\alpha_2...\alpha_n...}=\\
  =&\frac{i}{s}+
  \frac{1}{s}\Delta^{r_s}_{\beta_1\beta_2...\beta_n...}=
  \delta_i(\Delta^{r_s}_{\beta_1\beta_2...\beta_n...}).
\end{align*}
Hence, $\delta_i(x=\Delta^{r_s}_{\alpha_1\alpha_2...\alpha_n...})=
\frac{1}{s}x+\frac{i}{s}$ is an increasing affine map.
\begin{theorem}
  The graph $\Gamma_f$ of $f$ is a self-affine set
\[\Gamma_f=\bigcup\limits_{i=0}^{r}\varphi_i(\Gamma_f),\]
where $\varphi_i$ is defined by $\varphi_i: \begin{cases}
                   x'=\rho_i(x)=\Delta^{r+1}_{i\alpha_1(x)...\alpha_n(x)...}=\frac{i}{r+1}+\frac{1}{r+1}x,\\
                    y'=\delta_i(y)=\Delta^{r_s}_{i\alpha_1(x)...\alpha_n(x)...}=\frac{i}{s}+\frac{1}{s}f(x),
                 \end{cases}$ $i\in A_r$.
\end{theorem}
\begin{proof} Clearly $\Gamma_f=\Gamma_0\cup\Gamma_1\cup...\cup\Gamma_r$, where $\Gamma_i=\{M(x;y):~
x\in \Delta^{r+1}_i, y=f(x)\}$, $i=\overline{0,r}$.
We show that $\varphi_i(\Gamma_f)=\Gamma_i$.

Let $M(x;y)\in\Gamma_f$, where $x=\Delta^{r+1}_{\alpha_1...\alpha_n...}$ and $y=f(x)=\Delta^{r_s}_{\alpha_1...\alpha_n...}$. Consider the image point $M'(x';y')=\varphi(M(x;y))$. Then
$x'=\rho_i(x)=\Delta^{r+1}_{i\alpha_1\alpha_2\ldots}
\in\Delta^{r+1}_{i}$ and $y'=\delta_i(y)=f(x')$. Hence, $M'(x';y')\in \Gamma_i$, and therefore  $\varphi_i(\Gamma_f)\subset \Gamma_f$.

Now we show the converse inclusion.
Let  $M'(x';y')\in \Gamma_i$, that is, $x'=\Delta^{r+1}_{i\alpha_1\alpha_2...\alpha_n...}$ and
$y'=f(x')$. Then there exists $x=\Delta^{r+1}_{\alpha_1\alpha_2\ldots}$ such that $M(x,f(x))\in \Gamma_f$ and
$
M'(x',y')=\varphi_i(M(x,f(x))).
$
Consequently, $\Gamma_i\subset \varphi_i(\Gamma_f)$.

Combining both inclusions, we obtain
$\Gamma_i=\varphi_i(\Gamma_f)$.
\end{proof}
\begin{corollary}
  The self-affine dimension of the graph of $f$ is equal to $\frac{2\ln(r+1)}{\ln(r+1)s}$.
\end{corollary}
Indeed, the self-affine dimension is the solution of the equation
\[(r+1)\left|
         \begin{array}{cc}
           \frac{1}{r+1} & 0 \\
           0 & \frac{1}{s} \\
         \end{array}
       \right|^{\frac{x}{2}}=1.
\]
\begin{corollary}
  The following equality holds:
  \[\int\limits_{0}^{1}f(x)dx=\frac{r}{2(s-1)}.\]
\end{corollary}
Indeed, using the self-affinity of the graph of $f$, we obtain
\begin{align*}
  \int\limits_{0}^{1}f(x)dx=&\sum\limits_{i=0}^{r}\int\limits_{\Delta^{r+1}_{i}}f(x)dx=
  \sum\limits_{i=0}^{r}\int\limits_{0}^{1}(\frac{i}{s}+\frac{i}{s}f(x))d(\frac{i}{r+1}+\frac{x}{r+1})=\\
  =&\sum\limits_{i=0}^{r}(\frac{i}{s(r+1)}+\frac{1}{s(r+1)}\int\limits_{0}^{1}f(x)dx)=\frac{r}{2s}+
  \frac{1}{s}\int\limits_{0}^{1}f(x)dx.
\end{align*}
Hence
\[\int\limits_{0}^{1}f(x)dx=\frac{r}{2s}\cdot\frac{s}{s-1}=\frac{r}{2(s-1)}.\]
\section{Sets of function levels and their fractal properties}
Recall that the level set of the function $f$ at the value $y_0$ is defined by:
$$f^{-1}(y_0)=\{x\in [0,1]: f(x)=y_0\}.$$
\begin{theorem}
  If $r<2s-1$, then the function $f$  has level sets of different cardinalities:
  \begin{itemize}
\item[1)] singleton and finite;
\item[2)] countable and continuum.
\end{itemize}
Moreover, there exist level sets of positive Hausdorff--Besicovitch dimension.  
The set of all values whose level sets are finite or countable has
Hausdorff--Besicovitch dimension equal to
  $\frac{\ln(2s-r-1)}{\ln s}.$
\end{theorem}
\begin{proof}
  It is clear that the cardinality of the level set $f^{-1}(y_0)$, where $y_0=f(x_0)$ coincides with the cardinality of the set of distinct $r_s$-representations of the number $y_0\in [0,\frac{r}{s-1}]$.

   As noted above, the numbers $0$ and $\frac{r}{s-1}$ have unique $r_s$-representations: $f^{-1}(0)=\Delta^{r_s
  }_{(0)}$ and $f^{-1}(\frac{r}{s-1})=\Delta^{r_s}_{(r)}=1$. Hence, the corresponding level sets of the function $f$ are singletons.

   1. It is proved in~\cite{PratsVynnO} that the number $x=\Delta^{r_s}_{(c)}$, where $c\in \{0,r-s+2,r-s+3,...,s-2,r\}$, has a unique $r_s$-representation, whereas the number $x=\Delta^{r_s}_{(c)}$, where $c\in \{1,2,...,r-s,s,s+1,...,r-1\}$,
    has a continuum of distinct $r_s$-representations. Consequently, the corresponding level sets have the same cardinalities.
   
    2. We prove that the number $x=\Delta^{r_s}_{(c)}$, where $c\in \{r-s+1,s-1\}$,
     has a countable set of distinct $r_s$-representations.

Let $c=s-1$. Then 
\begin{equation}\label{num1}
    1=x_0=\Delta^{r_s}_{(s-1)}=\Delta^{r_s}_{[s-1]s(0)}=\Delta^{r_s}_{[s-1][s-1]...[s-1]s(0)},
  \end{equation}
  and, under the condition $r=2s-2$ 
  \begin{equation}\label{num2}
   1=x_0=\Delta^{r_s}_{(s-1)}=\Delta^{r_s}_{[s-1][s-2](r)}=\Delta^{r_s}_{[s-1][s-1]...[s-1][s-2](r)}.
     \end{equation}
  Thus, the number $x_0$ has infinitely many distinct $r_s$-representations. We show that no other representations exist.

First consider the case $r<2s-2$.

  The first digit of an $r_s$-representation of $x_0$ cannot be equal to $s-2$, since
  \[\Delta^{r_s}_{[s-2](r)}=\frac{s^2-3s+2+r}{s(s-1)}<\frac{s^2-3s+2+2s-2}{s(s-1)}=1.\]
 It also cannot be equal to $s+1$, because
\[\Delta^{r_s}_{[s+1](0)}=\frac{s+1}{s}>1.\]
  Hence, the first digit $\alpha_1(x_0)$ is either $s-1$ or $s$. If $\alpha_1(x_0)=s$,
  then $x_0=\Delta^{r_s}_{s(0)}$. Assume $\alpha_1(x_0)=s-1$. Then $x_0=\frac{s-1}{s}+x_1$, where $x_1=\frac{1}{s}$ and $\alpha_1(x_1)=\alpha_2(x_0)$. If $\alpha_2(x_0)=s$, then
  $x_0=\Delta^{r_s}_{[s-1]s(0)}$. Let $\alpha_2(x_0)\neq s$. The equality $\alpha_2(x_0)=s-2$ is impossible, since $\frac{1}{s}\Delta^{r_s}_{[s-2](r)}<\frac{1}{s}=x_1$,
and $\alpha_2(x_0)=s+1$ is also impossible because
    \[\frac{1}{s}\Delta^{r_s}_{[s+1](0)}=\frac{s+1}{s^2}>\frac{1}{s}=x_1.\]
  Therefore, $\alpha_2(x_0)=s-1$.

  Proceeding analogously, we obtain that under the condition
$\alpha_1(x_0)=\alpha_2(x_0)=s-1$ the next digit $\alpha_3(x_0)$ equals either $s$ or $s-1$. If $\alpha_3(x_0)=s$, then
$x_0=\Delta^{r_s}_{[s-1][s-1]s(0)}$.
Otherwise, the argument continues inductively. Hence, all possible $r_s$-representations of $x_0$ are exhausted by~\eqref{num1}.

  Now consider the case $r=2s-2$. If $s-1\neq \alpha_1(x_0)\neq s$, then $ \alpha_1(x_0)= s-2$ and $x_0=\Delta^{r_s}_{[s-2](r)}$.
   If $ \alpha_1(x_0)= s-1$ and $s-1\neq \alpha_2(x_0)\neq s$, then $x_0=\Delta^{r_s}_{[s-1][s-2](r)}$. 
   If $ \alpha_1(x_0)= \alpha_2(x_0) =s-1$ and $s-1\neq \alpha_3(x_0)\neq s$, then $x_0=\Delta^{r_s}_{[s-1][s-1][s-2](r)}$.
    Continuing in this way, we obtain precisely the representations in~\eqref{num2}. Hence, the set of all $r_s$-representations of $\Delta^{r_s}_{(s-1)}$ is countable as a union of two countable sets.

     Since the numbers $x$ and $x'=\frac{r}{s-1}-x$ have the same cardinality of sets of $r_s$-representations, because
  \[x'=\frac{r}{s-1}-x=\sum\limits_{k=1}^{\infty}\frac{r}{s^k}-\sum\limits_{k=1}^{\infty}\frac{\alpha_k(x)}{s^k}=
  \sum\limits_{k=1}^{\infty}\frac{r-\alpha_k(x)}{s^k}=
  \Delta^{r_s}_{[r-\alpha_1][r-\alpha_2]...[r-\alpha_k]...},\]
  the numbers $\Delta^{r_s}_{(s-1)}$ and $\Delta^{r_s}_{(r-s+1)}$ also have equally powerful families of representations.

   If the pair $(a, b)$ 
   admits an alternative replacement, then the level set
$f^{-1}(\Delta^{r_s}_{(ab)})$ is nowhere dense, has Lebesgue measure zero, and its Hausdorff--Besicovitch dimension is not less than $\frac12\log_s(r+1)$.

    Indeed, if $(c,d)$ is alternative to  $(a,b)$, i.e. $as+b=cs+d$, then the set $C=\{x: x=\sum\limits_{k=1}^{\infty}\frac{\alpha_k}{(r+1)^2}, \alpha_k\in\{as+b,cs+d\}\} $ is contained in the level set of $\Delta^{r_s}_{(ab)}$, and the dimension of the Cantor-type self-similar set $C$ is determined by
$2\cdot s^{-2x}=1$, hence equals $x=\frac12\log_s 2$.
    Finally, it is proved in~\cite{PratsVynnO}
     that the set of numbers with a unique $r_s$-representation has Hausdorff--Besicovitch dimension
$\frac{\ln(2s-r-1)}{\ln s}$.
Each such number generates a countable family of numbers having finitely many representations. By countable stability of dimension, the set of numbers with finitely many representations has the same dimension. The set of numbers with countably many representations is itself countable, and therefore its dimension is zero. 
    \end{proof}
\begin{theorem}
If $r>2s-2$, then all level sets of the function $f$, except for the levels $0$ and $\frac{r}{s-1}$,  are continuums.
\end{theorem}
\begin{proof} 
As noted above, the cardinality of the set of distinct $r_s$-representations of a number coincides with the cardinality of the corresponding level set of the function.

1.  Let $0<a<r$. Then the pair $(a,a)$ admits an alternative replacement.  
Indeed, if $a\le r-s$, then $$\frac{a}{s}+\frac{a}{s^2}=\frac{r-s}{s}+\frac{a+s}{s^2}.$$ If $a>r-s$, then
  $$\frac{a}{s}+\frac{a}{s^2}=\frac{a+1} {s}+\frac{a-s}{s^2}.$$  
  Hence, if some $r_s$-representation of a number contains infinitely many pairs of identical consecutive digits different from $0$ and $r$, then this number has a continuum of distinct $r_s$-representations, since each such pair admits an alternative replacement.

2. Consider the number $x_0=\Delta^{r_s}_{c(r)}$, where $c\neq r$. Then
\[\Delta^{r_s}_{c(r)}=\Delta^{r_s}_{[c+1][r-s](r)}=\Delta^{r_s}_{[c+1][r-s+1][r-s](r)}=
\Delta^{r_s}_{[c+1][r-s+1]... [r-s+1][r-s](r)}=\Delta^{r_s}_{[c+1](r-s+1)}.\]
The pair $(r-s+1,r-s+1)$ admits the alternative replacement $(r-s,r)$. Therefore, infinitely many alternatives arise, which implies that the set of distinct $r_s$-representations of $x_0$ is a continuum.

3. Consider $x_0=\Delta^{r_s}_{c(0)}=\Delta^{r_s}_{[c-1](s-1)}$, $c\neq 0$. It has a continuum set of distinct $r_s$-representations, since the pair $(s-1,s-1)$ admits the alternative replacement $(s-2,2s-1)$.

4. 
Let some $r_s$-representation of the number $x_0$ have no period $(c)$ and no infinite number of pairs of identical consecutive digits. Then, due to the finiteness of the alphabet, this representation of the number contains some ordered triple of digits $(a,b,c)$, where $a\neq b$, $b\neq c$ occurring infinitely many times as consecutive digits. We consider the possible cases..

4.1. If $b=0$, then for the pair $(a,b)$ the alternative is  $(a-1,s)$, since $a\ne 0$.

 4.2. If $a=0$, then $b\ne 0$, and for the pair $(b,c)$ the alternative is  $(b-1,c+s)$ when $c<s$, while for $c>s-1$, the alternative is $(b+1,c-s)$.
 
 In each case, alternative replacements generate infinitely many branches of representations, which yields a continuum of distinct $r_s$-representations. Consequently, all nontrivial level sets of $f$ are continuums.
\end{proof}
\begin{theorem}
Almost all (in the sense of Lebesgue measure)
level sets of the function $f$ are fractal (that is, have fractional Hausdorff--Besicovitch dimension) or anomalously fractal (that is, are uncountable and have Hausdorff dimension zero).
\end{theorem}
\begin{proof} As shown in~\cite{Pratratushmatst,PratsVynnO}, under the condition $r<2s-1$, almost all numbers in the interval $[0,\frac{r}{s-1}]$ admit a continuum set of distinct $r_s$-representations. Hence, almost all level sets of the function $f$ are uncountable. 

Every uncountable set has either positive Hausdorff--Besicovitch dimension or dimension zero. In the latter case, the set is called \emph{anomalously fractal}, since its topological dimension is zero and does not coincide with its fractal dimension.  

Moreover, it is proved in the same works that, under the condition $r<2s-2$, the set of numbers having a unique $r_s$-representation has Hausdorff--Besicovitch dimension
 $\frac{\ln(2s-r-1)}{\ln s}$.
\end{proof}

\end{document}